\documentclass{amsart} 
\usepackage{amsmath,amssymb,amscd,hyperref} 

\newcommand{\be}{\begin{enumerate}}
\newcommand{\ee}{\end{enumerate}}

\newcommand{\psn}{{\mathbb P}^{n}}
\newcommand{\psm}{{\mathbb P}^{m}}
\newtheorem{thm}{Theorem}[section]

\newtheorem{rem}[thm]{Remark}
\newtheorem{lem}[thm]{Lemma}
\newtheorem{define}[thm]{Definition}
\newtheorem{cor}[thm]{Corollary}

\newtheorem{prop}[thm]{Proposition}
\newtheorem{example}[thm]{Example}

\newcommand{\lra}{\longrightarrow}

\newcommand{\pf}{\begin{proof}}
\newcommand{\epf}{\end{proof}}

\newcommand*{\longhookrightarrow}{\ensuremath{\lhook\joinrel\relbar\joinrel\rightarrow}}

\begin{document}  

\date{today} 

\title{Ubiquity of complete intersection liaison classes} 
\author[]{Mark R. Johnson and Paolo Mantero$^{1}$}
\address{Department of Mathematical Sciences,
University of Arkansas, Fayetteville AR 72701} 
\email{markj@uark.edu}

\address{Department of  Mathematics, University of California Riverside, 900 University Ave.,  Riverside, CA 92521} 
\email{mantero@math.ucr.edu\newline
\indent{\it URL:} \href{http://math.ucr.edu/~mantero/}{\tt http://math.ucr.edu/$\sim$mantero/}}
\thanks{AMS 2010 {\em Mathematics Subject Classification:} 13C40, 14M06, 14M10.}
\thanks{$^1$ The second author gratefully acknowledges the support of an AMS-Simons Travel Grant.}
 
\begin{abstract}
In this paper, we provide constructions to enumerate large numbers of CI-liaison classes.   To this end, we introduce a liaison invariant and prove several results concerning it, notably that it commutes with hypersurface sections.   This theory is applied to the CI-liaison classes of ruled joins of projective schemes, yielding strong obstructions for such joins to lie in the same liaison class.   A second construction arises from the actions of automorphisms on liaison classes,  allowing the enumeration of many liaison classes of perfect ideals of codimension at least three. 
\end{abstract} 
 
\maketitle

\section{Introduction}
Two proper ideals $I$ and $J$ of a local Gorenstein ring are  directly linked if 
there is a complete intersection $(\underline{\alpha}) \subseteq I\cap J$ such that 
$I = (\underline{\alpha}) :J$ and $J = (\underline{\alpha}) :I$.  This relation generates an equivalence 
relation on the (unmixed) ideals called (complete 
intersection) {\em linkage} or {\em (CI-)liaison}.   
An ideal  is {\em licci}  if it is in the linkage class of a 
complete intersection.    Perfect ideals of codimension two are licci, but this fails 
for codimension at least three.  Given that there are non-licci Cohen-Macaulay ideals, it is natural 
to ask about the different linkage classes.  For example, it is known that there
are infinitely many liaision classes of  ACM curves in $\mathbb P^{4}$ (e.g. \cite{U}). 
However, relatively little of a more general nature seems to be known in the literature.   
In this paper, we construct large families of liaison classes of Cohen-Macaulay ideals of any codimension at least three.

The idea is to study the join of two ideals in a regular local ring, generalizing the idea of the (ruled) join of projective subschemes. 
We show that, when some of the individual ideals are themselves not licci, 
distinct such join ideals must lie in distinct liaison classes 
(Corollary~\ref{thm2}).   Applying this to the ruled join,
we prove that there are at least as many CI-liaison classes of codimension  
$c+3$ ACM subschemes in ${\mathbb P}^{n+5}$ as there are  
generic complete intersection ACM subschemes of codimension $c$ in $\psn$ 
(Corollary~\ref{ACM}).  
Studying the effect of automorphisms on liaison classes, we similarly use 
non-licci ideals to construct many liaison classes in power series over a field, whenever the codimension is 
three or more (Proposition~\ref{aut}).  

To distinguish these liaison classes we introduce an invariant, essentially the ideal
generated by the entire liaison class of the ideal.  Its study, which begins in Section 2, is influenced by the work of Polini and Ulrich \cite{PU} on ideals that are maximal in their linkage class.  The main work of the paper is then to understand this invariant sufficiently in the case of a join (and somewhat more generally) (Theorem~\ref{transversal}), although we are not able to give an exact relation in terms of the given ideals, except for the case of hypersurface sections 
(Theorem~\ref{hyp}). As an application, we obtain strong obstructions for join ideals to lie in the same linkage class, which also yield sufficient conditions for an ideal to be licci (although more of theoretic interest). For instance we prove the following characterization of licci ideals (Theorem \ref{licci}): Let $I$ be an unmixed ideal, and $X$ and $Y$ two new variables, then $(I,X)$ and $(I,Y)$ lie in the same linkage class if and only if $I$ is licci. Also, the theory developed in the paper provides a simple way to construct new ideals that are maximal in their linkage classes from old ones (Corollaries \ref{maxprop} and \ref{maxim}). These ideals have the property that most of their direct links are equimultiple of reduction number one and thus have Cohen-Macaulay blow-up algebras \cite{PU}. 

The results of the paper show a striking difference between complete intersection liaison and {\em Gorenstein liaison}, the latter theory defined by using links defined by Gorenstein ideals, rather than complete intersections.   In fact, 
 our work was inspired by a recent paper of Migliore and Nagel \cite{MN}, who show that any reduced ACM subscheme 
of $\psn$ becomes glicci (in the $G$-liaison class of a complete intersection)
when viewed as a subscheme of $\mathbb P^{n+1}$.   (Our join constructions generalize this viewpoint.)  
More precisely, they show that if $I$ is CM and generically Gorenstein then $(I,X)$ is glicci.  The aforementioned Theorem 3.4 shows, on the other hand, that the generic hypersurface sections $(I,X)$ and $(I,Y)$ are not even in the same CI-liaison class when $I$ itself is not licci.

Our constructions serve to indicate some of the difficulties in working with complete intersection liaison when the codimension is greater than two, and, in some sense, to explain the advantages of Gorenstein liaison.

\section{Preliminaries} 
  
Throughout this work, $(R,m,k)$ will denote a local Gorenstein ring with infinite residue field 
$k$ and $I$ denotes a Cohen-Macaulay (CM) $R$-ideal of positive codimension (unless specified otherwise). 
Recall that an $R$-ideal $I$ is Cohen-Macaulay if $R/I$ is Cohen-Macaulay.    

   We say that an ideal is {\em generically a complete intersection}   
 if it is a complete intersection  locally at each of its associated prime ideals. 

We say that a local ring $S$ is a {\em deformation} of a local ring $R$ if there is an $S$-regular sequence 
$\underline{x}\subseteq S$ such that $S/(\underline{x})\cong R$.  If in addition, $J\subseteq S$ and $I\subseteq R$ are ideals, 
$\underline{x}$ is regular on $S/J$ and $JR = I$, we say that $(S,J)$ is a deformation of $(R,I)$,
or just $J$ is a deformation of  $I$ if the rings are understood.   When $R$ is a $k$-algebra, we always
assume that deformations are also $k$-algebras, and we say that $S$ is a $k$-deformation. An ideal $I$ is {\em deformable to a generic complete intersection} if there exists a deformation $(S,J)$ of $(R,I)$ where $J$ is generically a complete intersection.
 
In the sequel, by linkage we always mean complete intersection linkage. Our main tools to study linkage are generic and universal linkage.  \begin{define}\label{defgeneric}
Let $f_1,\dots , f_n$ be a generating set for an ideal $I$ of grade $g>0$. Let $X$ be a generic
$n \times g$ matrix of variables  over $R$ and  $\underline{\alpha}=\alpha_1,\dots\alpha_g$ be the regular sequence in $R[X]$ defined as $$\alpha_i=\sum_{j=1}^nX_{ij}f_j\qquad \mbox{ for all }i=1,\ldots,g.$$ 

The $R[X]$-ideal $L_1(\underline f)$=$(\underline{\alpha})R[X]:_{R[X]} IR[X]$ is called the {\em first generic link} of $I$.
 Let $R(X)=R[X]_{mR[X]}$. The $R(X)$-ideal $L^1(\underline f)=L_1(\underline f)_{mR[X]}$ is called the  {\em first universal link} of $I$.
\end{define}

Huneke and Ulrich proved that these notions are essentially independent of the chosen generating set \cite[2.11]{SOL}, hence we will write $L_1(I)$ and $L^1(I)$ without referring to a specific generating set of $I$.

For $e\geq 2$, one defines inductively the $e$-th generic link of $I$ as $L_e(I)=L_{1}(L_{e-1}(I))$. Similarly, the $e$-th universal link is defined as $L^e(I)=L^{1}(L^{e-1}(I))$. It can be checked that $L_e(I)$ is linked to $IR[X]$ in $e$ steps, and $L^e(I)$ is linked to $IR(X)$ in $e$ steps or is the unit ideal.
We refer to \cite{SOL} for further information concerning the basic facts on generic and universal linkage we use in the sequel.
    
To begin with, we will need to extend some results from \cite{PU}.
 
\begin{lem} \label{ull}   Let $C$ be an $R$-ideal.  
The following conditions are equivalent.
\be  
\item [(a)] For some (every) $e$-th generic link $L_{e}(I)\subseteq R[X]$  one has $L_{e}(I)\subseteq CR[X].$
\item [(b)] For some (every) $e$-th universal link $L^{e}(I)\subseteq R(X)$  one has $L^{e}(I)\subseteq CR(X).$ 
 
\ee
\end{lem}
\pf  The fact that each condition is independent of the choice of generic or universal link follows 
similarly as in the proof of \cite[1.2]{PU}.   Since clearly (a)$\implies$(b),
it remains to verify that (b)$\implies$(a).  Let $L_{e}(I)\subseteq R[X]$ be any generic link. 
To show the containment $L_{e}(I)\subseteq CR[X]$, it suffices to show this locally at every associated prime of $CR[X]$.  Such an associated prime $q$ is extended from an associated prime $p$ of $C$ in $R$.  It follows that $L_{e}(I)_{q} = L_{e}(I)_{pR[X]}$ is a further localization of the universal link 
$L^{e}(I)$, so by hypothesis is contained in $CR(X)_{q} =CR[X]_{q}$.  This completes the proof. 
\epf

\begin{prop} \label{PUprop} 
Let $e\ge 1$ be an integer.
\be
\item [(a)] $L^{e}(I) = (1)$ if and only if $I$ is linked to a complete intersection in $e-1$ steps.
\item [(b)]  Let $C$ be a proper $R$-ideal.   Then
 $L^{e}(I)\subseteq CR(X)$ 
 if and only if  $L^{e}(I) \neq (1)$ and  
 every $R$-ideal that is linked to $I$ in $e$ steps is contained  in $C$.  
\ee
\end{prop}
\pf Let   $I = I_{0}\sim I_{1}\sim\cdots\sim I_{e}$ be any sequence of links.  
 By \cite[2.17]{SOL} there exists a prime $q\in$ Spec $R[X]$ containing the maximal ideal $m$ of $R$, such that $(R[X]_{q}, L_{j}(I)_{q})$ is a deformation of $(R,I_{j})$ for every $1\le j\le e.$
Now suppose that $I_{e-1}$ is a complete intersection.   Then so is $L_{e-1}(I)_{q}$, hence
either $L^{e-1}(I)$ is a complete intersection or $L^{e-1}(I)$ is the unit ideal; in either case  
$L^{e}(I)$ is the unit ideal \cite[2.13]{SOL}.   
Conversely, suppose that $L^{e}(I)$ is the unit ideal.  Let $s$ be the largest integer, $0\le s < e$ 
with $L^{s}(I)\neq (1)$.  Then $L^{s+1}(I) = (1)$, hence $L^{s}(I)$ is a complete intersection. 
By \cite[2.4]{AL} there exists a sequence of links $I_{j}$ as above, for $1\le j\le s$, with 
$\mu(I_{s})  = \mu(L^{s}(I))$.  Hence $I_{s}$ is a complete intersection, and $I$ is linked 
to a complete intersection in $s$ steps, and therefore also  $e-1$ steps.   This proves (a). 

To show (b), suppose first that $L^{e}(I)\subseteq CR(X).$   Then clearly $L^{e}(I)\neq (1)$, and by Lemma~\ref{ull} we have that
$L_{e}(I)\subseteq CR[X]$.  Hence for any sequence 
of links $I_{j}$ as above, by specialization it follows that $I_e\subseteq C$.

For the converse, 
we suppose that $L^{e}(I)\nsubseteq CR(X)$, hence $L_{e}(I)\nsubseteq CR[X]$. 
If $X$ has $N$ entries, and $A\in R^{N}$, we set $\overline{A}$ for the image of $A$ in $k^{N}$, and write $\pi_{A}$ for the $R$-algebra epimorphism 
$\pi_{A}:R[X]\rightarrow R$ that sends $X$ to $A$. 
Now the proof of \cite[1.4]{PU} shows that there is a dense open subset $U$ of $k^{N}$ such that for any $A\in R^{N}$ for which $\overline{A}\in U$, we have $\pi_{A}(L_{e}(I))\nsubseteq C$.

On the other hand, since $L^{e}(I)\neq (1)$, by \cite[2.2]{AL}, there is also a dense open subset 
$V$ of $k^{N}$ such that for all $A\in R^{N}$ for which $\overline{A}\in V$, 
there is a sequence of links of $R$-ideals $I = \pi_{A}(L_{0}(I))\sim \pi_{A}(L_{1}(I))\sim \cdots \sim
\pi_{A}(L_{e}(I))$.   By hypothesis, since the latter ideal is linked to $I$ in $e$ steps, 
$ \pi_{A}(L_{e}(I))\subseteq C$.  
Therefore we obtain the required contradiction, for any $A\in R^{N}$ with $\overline{A}\in U\cap V$. 
\epf
   
As  a consequence of Proposition~\ref{PUprop} we obtain the following result.
 
\begin{thm} \label{PUcor}   
The following conditions are equivalent for an $R$-ideal $C\subsetneq m$.  
\be 
\item[(a)]  $L^{e}(I)\subseteq CR(X)$ for every $e\ge 1$  
\item[(b)]   
  $C$ contains every ideal in the linkage class of $I$.
\ee
\end{thm} 
\pf This follows directly from Proposition~\ref{PUprop} once we verify that a licci ideal
$I$ cannot satisfy condition (b).  Indeed, in that case, $C$ would contain every ideal in the linkage
class of $I$, and hence every ideal in the linkage class of a complete intersection.  But all complete 
intersections of the same codimension $g$ belong to the same linkage class, so $C$ contains 
 every complete intersection ideal of codimension $g > 0$.  In particular, $C$ would contain every 
nonzerodivisor of $R$, and therefore $C = m$.  This case is excluded, so this completes the proof. 
\epf

One says that an ideal is {\em maximal in its linkage class} if it contains every ideal in its linkage class. 

According to our previous result, we get the following characterization: an ideal 
 $I\neq m$ is maximal in its linkage class if and only if 
$L^{e}(I)\subseteq IR(X)$ for every $e\ge 1$.   This is a weaker version of a theorem of 
Polini and Ulrich, who show  that this condition also is equivalent to just 
$L^{1}(I)\subseteq IR(X)$ (\cite[1.4]{PU}).  However, in their statement, they omit the condition
$I\neq m$, which is essential when $R$ is regular.   
\bigskip

The previous results motivates the following definition. 

\begin{define}
Let $I$ be an unmixed $R$-ideal. We define 

\begin{equation*} 
 \int I =  
\begin{cases} 
     \text{sum of all ideals in the linkage class of }  I & \text{if I is not licci}\\
     \text{unit ideal} & \text{if I is licci} 
\end{cases} 
\end{equation*}
\end{define}
If we wish to specify the ring, we use the notation $\int_{R} I$. 

We record next some basic properties of this notion.

\begin{thm} \label{basic} Let $I$ be a CM $R$-ideal.  
\be
\item[(a)] $\int I$ is an $R$-ideal containing $I$.  
\item[(b)] If $I$ and $J$ are in the same linkage class then $\int I = \int J$.
\item[(c)]  $\int I$ is the unique smallest ideal with  $L^{e}(I)\subseteq (\int I)R(X)$ for every $e\ge 1.$  
\item[(d)]   If $T$ a flat  local Gorenstein  extension of $R$ with $IT\neq T$ then $\int_{T} IT = (\int_{R} I)T$.  In particular, for all $p\in V(I)$, $\int_{R_{p}} I_{p} = (\int_{R} I)_{p}$.
\ee
\end{thm}
\pf Parts (a) and (b) are clear, while (c) follows from Theorem~\ref{PUcor} and Proposition~\ref{PUprop}(a).   We show (d).  
Let 
$$I = I_{0} \sim I_{1}\sim \cdots \sim I_{n} = J$$ be
any sequence of links.   Suppose first that $I_{i}T \neq T$ for all $i$. 
Then $$IT = I_{0}T \sim I_{1}T\sim \cdots \sim I_{n}T = JT$$ 
is a sequence of links  in $T$.   Therefore 
$(\int_{R} I)T\subseteq \int_{T}IT$ in this case. 
 Now suppose that 
$I_{i}T = T$ for some $i$, and that $i\ge 1$ is the least integer with
this property.   If $I_{i-1}\sim I_{i}$ is linked via the complete intersection $(\underline{\alpha})$, then 
then $I_{i-1}T = ((\underline{\alpha}):I_{i})T = (\underline{\alpha})T$ 
is a complete intersection.  Since by the first part of the argument, 
$$IT = I_{0}T \sim I_{1}T\sim \cdots \sim I_{i-1}T$$ 
is a sequence of links, in this case $IT$ is licci, and the containment holds 
again.    To show the reverse containment, we may assume that $I$
is not licci.   Let $L^{e}(I)\subset R(X)$ be an $e$-th universal link and $T' = T(X)$.
Then by (c), 
$L^{e}(IT) = (L^{e}(I))T' \subseteq (\int I)T' = ((\int I)T)T'$.  
Hence by the minimality property (c) again, we conclude that 
$\int IT \subseteq (\int I)T.$
\epf

\begin{lem} \label{maxrem}  
$\int I = I$ if and only if $I$ is maximal in its linkage class and is not the maximal ideal
of a regular local ring.  
\end{lem} 
\pf  It suffices to show 
that the only licci ideal that is maximal in its linkage class is the maximal ideal of a regular 
local ring. Indeed, from the proof of Theorem \ref{PUcor} it follows that any such ideal has to be the maximal ideal of the ring. Being licci, it also has finite projective dimension (\cite[2.6]{PS}) so the ring is also regular.
\epf
  
The {\em nonlicci locus} is defined by 
$$\text{Nlicci}(I) = \{p\in V(I) \ \vert\  I_{p} \text{ is not licci in } R_{p}\}.$$ 
It is a Zariski closed subset of Spec $R$ (\cite[2.11]{AL}).  
The fact that $\int$ commutes with localization implies the following 
sharper version of this fact, which partially explains the separation of the non-licci case.

\begin{prop} \label{Nliccidefideal} 
$\textnormal{Nlicci}(I) = V(\int I).$
\end{prop}
\pf We have $p\in $ Nlicci$(I)$ if and only if $ \int_{R_{p}} I_{p} \neq R_{p}$,
or equivalently $ (\int_{R} I)_{p} \neq R_{p}$  by Theorem~\ref{basic}(d),
so this is equivalent  to $p\in V(\int_{R} I).$
\epf
  
We now compute $\int I$ in a couple of situations. 
\begin{thm}\label{PU}[\cite[2.10]{PU},\cite[1.1]{W}] Let $I$ be an unmixed ideal of $R$ of codimension $g \ge 2$. 
Suppose that either \be 
\item [(a)] $I$ is generically a complete intersection, $g\ge 3$ and $t\ge 2$, or 
\item [(b)] $I$ is reduced, and not a complete intersection  at any minimal prime. 
\ee Then the $t$-th symbolic power $I^{(t)}$ of $I$ is maximal in its linkage class.   In particular, if $t\ge 2$ then  $\int I^{(t)} = I^{(t)}$. 
\end{thm}

\begin{thm}[{\cite[3.8]{UU}}] \label{shift} Let  $R' = k[x_{1},...,x_{n}]$ and let $I'$ be 
a homogeneous CM $R'$-ideal of codimension $g$ whose graded minimal free resolution 
has the form $$0\lra \oplus R'(-n_{gi})\lra \cdots \lra \oplus R'(-n_{1i}) \lra I' \lra 0$$
with max $\{n_{gi}\} \le (g-1)min \{n_{1i}\}$.  Suppose that $I$ has initial degree $d$.  Then with 
$R = k[x_{1},...,x_{n}]_{(x_{1},...,x_{n})}$, $m = (x_{1},...,x_{n})R$ and $I = I'R$ we have 
$$\int I \subseteq m^{d}.$$
\end{thm}

\begin{example} \label{ex1} Let $R = k[x,y,z,w]_{(x,y,z,w)},   m = (x,y,z,w)$, and let $I = (x^{2},xy,y^{2},z^{2},zw,w^{2})$. 
Then $\int I = m^{2}$.  
\end{example}
\pf Let $R' = k[x,y,z,w]$ and consider the homogeneous ideal $I'$ generated by the corresponding
forms in $R'$ that generate $I$.  Then the graded minimal free resolution has the form 
$$0\lra R'^{4}(-6)\lra \cdots \lra R'^{6}(-2)\lra I'\lra 0.$$
Since $I'$ has codimension $g = 4$, the condition of Theorem~\ref{shift} is satisfied, hence 
$\int I \subseteq m^{2}$.  On the other hand, linking via the regular sequence $x^{2},y^{2},z^{2},w^{2}$ clearly yields a link containing the product 
$(x,y)(z,w)$.  Since the mixed terms $xy, zw$ already belong to $I$, 
it follows that $m^{2}\subseteq \int I$, so equality holds.
\epf
Note that in the above example $\int I =\int m^2=m^2$ but $I$ and $m^2$ are not in the same linkage class, by \cite[5.13]{SOL}.

The following result will be needed later in the paper.

\begin{lem}  \label{linksdeform} Let $(S,J)$ be a deformation of $(R,I)$.  Then 
$$\int I \subseteq (\int J)R.$$
\end{lem}
\pf Let
$$I = I_{0} \sim I_{1}\sim \cdots \sim I_{n}$$ be a sequence of links in $R$.  Then by \cite[2.16]{SOL} there is a sequence of links 
$$J = J_{0} \sim I_{1}\sim \cdots \sim J_{n} $$ in $S$ such that $(S,J_{i})$ is a deformation of $(R,I_{i})$ for every $i$.   The result therefore follows.
\epf

\section{Liaison of hypersurface sections}

 The main result in this section is the following theorem, showing that the linkage invariant $\int$ is compatible with 
taking hypersurface sections.  
 
\begin{thm}[Hypersurface Section Formula]  \label{hyp} 
Let $x\in R$  be  regular on $R$  and on $R/I$.   Then 
$$\int (I,x) =  ( (\int I), x).$$
\end{thm}
  \pf 
We first show the containment  $``\subseteq$''.
 If $I$ is licci, the result is clear, 
so we may assume that $I$ is not licci. By 
Theorem~\ref{basic},   it suffices to show that   $$L^{e}((I,x))\subseteq ((\int I),x)R(Y)$$ for every $e\ge 1$.  
By \cite[2.2]{HU2}, 
$$L^{1}((I,x)) = (H, z)$$ 
where $H$ is an ideal directly linked to $IR(Y)$ and $z\in R(Y)$ is regular on $R(Y)$ and on $R(Y)/H.$  
Furthermore, $z\in (I,x)$.  We set $H = H_{1}$ and $z = z_{1}$.  
By induction on $e$, we conclude that, for every $e\ge 1$, there exists a universal link 
$L^{e}((I,x))$  in an extension of $R$ (which we denote again by $R(Y)$) with
$$L^{e}((I, x)) = (H_{e}, z_{e}),$$ with $H_{e}$ is linked to $H_{e-1}R(Y)$, and 
$z_{e}\in R(Y)$ is regular on $R(Y)$ and on $R(Y)/H_{e}.$ 
It follows that $H_{e}$ is linked to $IR(Y)$ in $e$ steps. 
Furthermore, we also have that $$z_{e}\in (H_{e-1},z_{e-1}) 
\subseteq (H_{e-1},H_{e-2},z_{e-2})\subseteq \cdots 
\subseteq (H_{e-1},...,H_{1},I,x).$$   
Now for every $i$, we have that $H_{i}\subseteq \int (IR(Y)).$  But by 
Theorem~\ref{basic}, $\int (IR(Y)) = (\int I) R(Y)$.
Therefore  $$L^{e}((I,x)) = (H_{e},z_{e}) \subseteq (H_{e},H_{e-1},...,H_{1},I,x)\subseteq
((\smallint I), x)R(Y)$$
and   the required containment follows. 

Now to show the equality, it suffices to show that $\int I \subseteq \int (I,x)$. 
If $I$ is licci then so is $(I,x)$ by \cite[2.3]{HU2}, so the result is clear, and again 
we may assume that $I$ is not licci.   

   We shall use a more precise description of the ideal $H$ described at the beginning of the proof, in the formula for the first universal link $L^{1}((I,x))$ of the hypersurface 
section.  A routine matrix argument (using the proof of \cite[2.2]{HU2}, cf. also the proof of \cite[2.3]{LP})
shows that one may take $H$ to be (the extension of) a first universal link $L^{1}(I)$ of $I$.   
   Hence, one has 
$L^{1}(I)\subseteq L^{1}((I,x))$ and by induction 
$L^{e}(I)\subseteq L^{e}((I,x))$ for every $e$. 
 Therefore, by Theorem~\ref{basic},
$L^{e}(I)\subseteq (\int(I,x))R(Y)$ for every $e$, and then again 
by the minimal property of $\int I$, we  conclude that    $\int I \subseteq \int (I,x)$. 
This completes the proof.  
\epf

\begin{cor} \label{cicase}
Let $\underline{x} \subseteq R$ be a sequence that is regular on $R$ and on $R/I$ and set $J = (\underline{x})$.   
Then $$\int(I+J) = (\int I) + J.$$ 
\end{cor}\smallskip

\begin{cor} \label{maxprop}  
Let $\underline{x}\subseteq R$  be a sequence that is regular on $R$  and on $R/I$.  
If $I$ is maximal in its linkage class then so is $(I, \underline{x})$.  
\end{cor}
\pf  
  By Corollary~\ref{cicase}, since $I$ is not the maximal ideal,  
$$\int (I,\underline{x}) =  (\int I) + (\underline{x})  = (I, \underline{x}),$$
  so the result follows.
\epf

We now give a hypersurface section characterization of licci ideals.  
The following result is an immediate consequence of Theorem~\ref{hyp} in the CM  case.  However we can prove this 
result in somewhat greater generality.  

\begin{thm}\label{licci}  Let $(R,m)$ be a local Gorenstein ring with an infinite residue field and let $I$
be an unmixed $R$-ideal.  Then $I$ is licci if and only if 
the ideals $(I,X)$ and $(I,Y)$ are in the same linkage class in  $R[X,Y]_{(m,X,Y)}.$
\end{thm}
\pf If $I$ is licci, then so is $(I,X)$ and $(I,Y)$, so these two ideals belong to the same 
linkage class.  Conversely, suppose that  $(I,X)$ and $(I,Y)$ are in the same linkage class.
Then $\int(I,X) = \int(I,Y)$ and therefore by Proposition~\ref{Nliccidefideal} (which does not require CM for this containment) 
$$\textnormal{Nlicci}((I,X)) \subseteq V(\int(I,X)) = V(\int(I,Y)).$$ 
Since $P = (m,X) \notin V(\int(I,Y))$ we obtain that $(I,X)_{P}$ is licci.  
In particular, the ideal $(I,X)_{P}$ is CM so we have $IR[X,Y]_{P}$ is CM, and
it follows by Theorem~\ref{hyp} that $IR[X,Y]_{P}$ is licci.  
But by descent \cite[2.12]{AL} we conclude that
$I$ is licci.  
\epf

\section{Liaison of joins} 

In this section we generalize the hypersurface section formula of the previous section.

Our main result be will in the regular case.  
Recall that if $R$ is a regular local ring, two $R$-ideals $I$ and $J$ are 
{\em transversal} if $I\cap J = IJ$.  (Geometrically, this condition implies that the
subschemes defined by $I$ and $J$ meet properly.)

\begin{thm}\label{transversal}  Let $R$ be a regular local ring containing an infinite field 
and let $I$ and $J$ be two transversal CM $R$-ideals, and assume 
that $J$ is deformable to a generic complete intersection.   Then 
$$\int (I+J)    \subseteq (\int I) +  J.$$ 
\end{thm} 

Before we begin the proof, we would like to discuss some of the consequences of this result. 

First, one should note that if $J$ is a complete intersection, then equality holds, by  Corollary~\ref{cicase}.  However, equality usually will not hold in Theorem~\ref{transversal}.  Indeed, we have the following immediate corollary.

\begin{cor}\label{transversal2}  Let $R$ be a regular local ring containing an infinite field
and let $I$ and $J$ be two transversal CM $R$-ideals that are both
deformable to a generic complete intersections.   Then 
$$\int (I+J)    \subseteq  I+J + [(\int I)  \cap (\int J) ].$$ 
\end{cor} 

Even  in this refined relation equality need not hold.   For example, if $I$ and $J$ are licci, then equality would mean that $I+J$ is licci, which is usually not the case (see, for instance, Example \ref{ex1} and, more generally, Theorem~\ref{J}).

In the case where $R$ is a regular local ring, we can then give a stronger version of Corollary 3.3 (where the ideal $J$ can be more general than a complete intersection).

\begin{cor}\label{maxim}  Let $R$ be a regular local ring containing an infinite field 
and let $I$ and $J$ be two transversal CM $R$-ideals, and assume 
that $J$ is deformable to a generic complete intersection.
If $I$ is maximal in its linkage class then so is $I + J$. 
\end{cor}
\pf  By Theorem~\ref{transversal},
$$I +J \subseteq \int(I+J) \subseteq  (\int I) + J  = I +  J $$
since $I$ is not the maximal ideal.  Hence 
equality holds, and we are done by Lemma~\ref{maxrem}.
\epf

The combination of Corollary \ref{maxim} and Theorem \ref{PU} then allows one to produce large classes of ideals that are maximal in their linkage classes.

As an application, we can characterize precisely when a transversal sum is licci. 

\begin{thm}\label{J}  Let $R$ be a regular local ring containing an infinite field,
and let $I$ and $J$ be two transversal $R$-ideals, 
 one of which is deformable to a generic complete intersection.  
Then $I+J$ is licci if and only both $I$ and $J$ are licci, and one of them is a complete intersection.
\end{thm}
\pf  
First assume both ideals are licci and one is a complete intersection. Since the ideals are transversal, by \cite[Lemma~2.2]{J2}, one has ${\rm ht}(I+J)={\rm ht}(I)+{\rm ht}(J)$. Since $J =(x_1,\ldots,x_h)$ is a complete intersection
and $I$ is CM, by induction on $h$ one has that $x_1,\ldots,x_h$ form a regular sequence on $R/I$. Then $I+J$ is 
a hypersurface section of a licci ideal, and therefore is licci (e.g. Theorem~\ref{hyp}). 
 For the converse, suppose that 
$I+J$ is licci, and that $I$   is deformable to a generic complete intersection. Since $I+J$ is CM, then so are $I$ and $J$ (see also discussion after Definition \ref{join}). 
 Then by Theorem~\ref{transversal} $$R = \int (I+J) \subseteq I + \int J,$$ 
hence $\int J = R$ and $J$ is licci.   
In particular, $J$ is deformable to a generic complete intersection (see Lemma~\ref{def-rem}),
hence by interchanging the roles, we also conclude that $I$ is licci. 
  The fact that one of $I$ or $J$ must be a complete intersection
now follows by \cite[2.6]{J}.
\epf

To prove Theorem~\ref{transversal}, we reduce to the situation where the sum $I+J$ is a join. 
By this we mean the following:  

\begin{define}\label{join}
let $R$ and $S$ be complete local noetherian $k$-algebras with residue field $k$.
We let $T = R\hat{\otimes}_{k} S$ be their complete tensor product over $k$. 
 Further, let $I$ be an $R$-ideal and $J$ be an $S$-ideal.  
We associate to this pair the $T$-ideal $K$ generated by the extensions of $I$ and $J$ to $T$.  
We denote this ideal by $K = (I,J)$.   We call $K$ the join of $I$ and $J$.
\end{define}

For example, if $S = k[[X]]$ is a power series algebra over $k$, then 
$T \cong R[[X]]$ and one can identify the join $(I,J)$ with the sum $IT+JT$ 
of extended ideals from the two natural subrings $R$ and $S$ of $R[[X]]$. 

We will routinely use the standard facts that the maps $R\rightarrow T$ are flat 
and that therefore if $R$ and $S$ are CM (resp. Gorenstein, regular) then so is $T$.  
Furthermore, $ T/K \cong R/I \ \hat{\otimes}_{k} \ S/J.$
\bigskip

{\em Proof of Theorem~\ref{transversal}}. 
 
We reduce  to the join case.  Without loss of generality, we may assume 
that $R$ is complete.   Indeed, if $\hat{R}$ is the completion of $R$,
then $I\hat{R}$ and $J\hat{R}$ are transversal $\hat{R}$-ideals, 
and $J\hat{R}$ is still deformable to a generic complete intersection, so if the
result is known in the complete case, then
$$
(\int( I+J))\hat{R}  
 = \int(I\hat{R}+J\hat{R}) \subseteq  
(\int I\hat{R}) + J\hat{R} = ((\int I) + J))\hat{R}.$$
and the result now follows for $R$ by faithfully flat descent.  
 
Write $R \cong k[[X]]$.  
Let $S = k[[Y]] \cong R$, and let $\phi:R \lra S$ be the 
$k$-algebra isomorphism sending $X$ to $Y$.  Let $\tilde{I} = I$
and let $\tilde{J} = \phi(J)$.    
Set $T = R\hat{\otimes}_{k} S \cong k[[X,Y]]$.  
  Then 
$(T, (\tilde{I},\tilde{J}))$ is a deformation of $(R,(I+J))$.    
Indeed, the $k$-algebra homomorphism $\pi:T\longrightarrow R$ 
with $\pi(X) = X$ and $\pi(Y) = X$ has kernel the regular sequence 
generated by the entries of $X-Y$ and $\pi((\tilde{I},\tilde{J})) 
 = I + J$.   If the result is known in the join case,  by
Lemma~\ref{linksdeform}, we have
 $$\int(I+J)   \subseteq  (\int (\tilde{I},\tilde{J}))R  
\subseteq  (( \int \tilde{I}), \tilde{J})R  
 =    (\int I)  + J. $$

Thus to complete the proof, we may assume that $I+J$ is a join. 
In this case, we are able to prove the result under the more general setting
that $R$ and $S$ are Gorenstein $k$-algebras.

\begin{prop}  \label{max} 
Let $I$ be a CM $R$-ideal and let $J$ be a CM $S$-ideal that is deformable to a generic 
complete intersection.    Then 

$$\int(I,J) \subseteq ( (\int I), \ J). $$ 
\end{prop} 
\pf     
 
Without loss of generality, we may assume that $I$ is not licci.
  Let $K = (I,J)$.  
By hypothesis, there is a $k$-deformation $(S',J')$ of $(S,J)$ with $J'$ is generically a 
complete intersection.  By  completing $S'$ if necessary we may assume 
that $S'$ is complete.    If $T' = R\hat{\otimes}_{k} S'$ and
$K' = (I, J')$ is the corresponding join, then 
$(T', K') $ is a  $k$-deformation of $(T, K)$. 
Indeed, by induction,  it suffices to show this when dim $S'$ = dim $S +1$, and
if $a\in S'$ is regular on $S'$ and $S/J'$ with  
$(S'/(a), (J',a)/(a))\cong (S,J)$,  then $a\in T'$ is regular on $T'$ and on
$T'/K' \cong R/I \ \hat{\otimes}_{k} \ S'/J'$ and  
$(T'/(a), (K',a)/(a)) \cong (R \hat{\otimes}_{k} S, (I,J',a)/(a)) \cong (T,K).$  
Now if the result is known for $K'$ then by 
Lemma~\ref{linksdeform}, 
$$\int K \subseteq (\int_{T'} K')T \subseteq ((\int I), J')T = ((\int I), J).$$   
Therefore we may assume without loss of generality that $J = J'$ is generically a complete 
intersection.
   
We may also assume that $k$ is algebraically closed.  
Indeed, if the result is known in this case, we let $\overline{k}$ be the algebraic closure 
of $k$, and replace $R$ and $S$ by $R' = R\hat{\otimes}_{k} \overline{k}$, 
and $S' = S\hat{\otimes}_{k}\overline{k}$;  in this case 
$J ' = JS'$ is still generically complete intersection and $I' = IR'$ is still not licci.    
Thus by faithful flatnessness, the containment descends from $\overline{k}$ to $k$. 

    Let $C =\int I$.  To verify that $\int K \subseteq 
(C, J)$, it suffices to show this locally at 
every associated prime of $(C,J)$.   Since this latter ideal is a join, by flatness we have 
\begin{eqnarray*}  
Ass (T/(C,J)) & = & Ass (R/C \  \hat{\otimes}_{k} \ S/J) \\
&  = & \bigcup_{p\in Ass (R/C)} Ass ((R/C \ \hat{\otimes}_{k} \ S/J)/p(R/C) \ \hat{\otimes}_{k} \ S/J)\\
&  = & \bigcup_{p\in Ass (R/C)} Ass ( R/p \ \hat{\otimes}_{k} \  S/J)\\
& = &  \bigcup_{p\in Ass (R/C)}\bigcup_{q\in Ass (S/J)} Ass ( R/p \ \hat{\otimes}_{k}\ S/q)\\  
& = &  \bigcup_{p\in Ass (R/C)}\bigcup_{q\in Ass (S/J)} \{(p,q)\},
\end{eqnarray*}
the last equality since $k$ is algebraically closed \cite[7.5.7]{EGA}.   In particular, every such associat{}ed prime is contained a prime $Q = (m, q)$.   
To verify the containment then, it suffices to verify the containment locally at every such $Q$. 

 Locally at $Q$,  $JT_{Q}$ is a complete intersection, so $K_{Q}$ is a hypersurface section 
of $IT_{Q}$. 
Therefore, by Corollary~\ref{cicase}, 
$$\int_{T_{Q}} K_{Q}  =   (\int_{T_{Q}} IT_{Q}) + JT_{Q} =  ((\int I), \ J)T_{Q}.$$ 
This establishes the claim, and the proof is complete. 
\epf
  
The condition in Theorem~\ref{transversal} and its corollaries, that an
 ideal admits a generic
complete intersection deformation, is a somewhat rather weak requirement.   
Other than generic complete intersections themselves,  for example, in a regular ring, this includes 
any (CM) monomial ideal (``polarization''), or a determinantal ideal of the expected codimension.  This also holds when the ideal is linked to a generic complete intersection, by the following remark. 

\begin{lem} \label{def-rem} Let $R$ be a local Gorenstein ring and let $I$ be a CM $R$-ideal that is in the linkage class of a generic complete intersection.
 Then $I$ has a deformation to a generic complete intersection.
 \end{lem}
\pf  
   If $I$ can be linked to an ideal $J$ in $n$ steps, 
then by \cite[2.17]{SOL} there is a generic link $L_{n}(J)\subseteq R[X]$ of $J$ and a  prime 
$Q$ of $R[X]$ such that  $L_{n}(J)_{Q}$ is a deformation of $I$.   Hence by induction it
suffices to show that the property of being a generic complete intersection is 
preserved from an ideal to a first generic link, which is proved in \cite[2.5]{DCG}.  
\epf

We next apply the join result Proposition~\ref{max} to give strong obstructions for two joins to belong to the same linkage class.

\begin{cor} \label{thm1}  
Let $I$ and $I'$ be CM $R$-ideals,  let $J$ and $J'$ be CM  $S$-ideals,
and suppose that 
$(I, J)$ lies in the same linkage class  as $(I',J')$.
\be \item[(a)]  Suppose that $J$ can be deformed to a generic complete intersection 
and that $J'\nsubseteq J$.  Then $I$ is licci. 
\item [(b)] Suppose that  $J \neq J'$  are both deformable to generic complete intersections.
Then either $I$ or $I'$ is licci.  
\item [(c)] Suppose that all the ideals are   deformable to generic complete intersections
and that  $(I,J) \neq (I',J')$.  Then one of the ideals is licci.  
\ee
\end{cor}  
\pf If $I$ is not licci then Proposition~\ref{max} implies that $(I',J')\subseteq (m,J)$, which can only occur when  $J'\subseteq J$. The rest follows by symmetry. 
\epf

For generic complete intersections one can show a slightly stronger form of the previous corollary.

\begin{cor}\label{thm2}
Let $I$ and $I'$ be CM $R$-ideals,  let $J$ and $J'$ be CM  $S$-ideals 
of the same codimension that are generic complete intersections. 
If
$(I, J)$ lies in the same linkage class  as $(I',J')$ and $J\neq J'$ 
   then $I$ and $I'$ are licci. 

In particular, if in addition all the ideals are generic complete intersections
and $I\neq I'$,  then all of the ideals are licci.  
\end{cor}  
\pf By Corollary~\ref{thm1}, it suffices to show that if $I'$ is licci then so is $I$. 
Suppose that $I$ is not licci.  Since $K = (I,J)$ belongs to the same linkage class
as $K' = (I',J')$, we have $\int K = \int K'$ and hence Nlicci$(K)$ = Nlicci$(K')$.
Let $q\in V(J)$ be a minimal prime and set $Q = (m,q).$  Since $R\hookrightarrow T_Q$ is a faithfully flat and $I$ is not licci, then by \cite[2.12]{AL} $I_Q$ is not licci.  Similarly, since $J_q$ is a complete intersection, then $JT_{Q}$ is a complete intersection, thus $K_{Q}$ is a hypersurface section of $IT_{Q}$. Since $I_Q$ is not licci, then $K_Q$ is not licci, i.e. $Q\in$ Nlicci($K)$.
Hence $Q\in $ Nlicci($K')$ 
and therefore $I'$ is not licci, since $J'T_{Q}$ is a complete intersection.
  \epf

  In the remainder of this section, we apply these results to the 
CI-liaison classes of ruled joins.

Let $k$ be an algebraically closed field.    
Given closed ACM subschemes $X\subseteq\psn$ and $Y\subset\psm$, we denote the ruled join of $X$ and $Y$ by $J(X,Y)$.  
This is a subscheme of $\mathbb P^{n+m+1}$ 
consisting of the union of 
all lines joining points of $X$ and $Y$, considered 
as embedded in $\mathbb P^{n+m+1}$ as disjoint 
subschemes in the natural way.  
  If $I_{X}\subseteq k[x_{0},..., x_{n}]$ and $I_{Y}\subseteq k[y_{0},..., y_{m}] $ are
the homogeneous ideals of $X$ and $Y$, then $J(X,Y)$ has ideal $(I_{X},I_{Y})\subseteq k[x_{0},...,x_{n},y_{0},...,y_{m}]$.

\begin{prop} \label{joinprop1} Let $X, X'\subseteq\psn$ be ACM subschemes, one of which is not licci 
locally at the vertex of its affine cone, and let
  $Y \neq Y'\subseteq\psm$ be generic complete intersection ACM subschemes.  
Then
$J(X,Y)$ is not in the same CI-liaison class as $J(X',Y')$ in $\mathbb P^{n+m+1}$.   
\end{prop}
\begin{proof} This follows immediately from Corollary~\ref{thm2}.
\end{proof}
 
We use this to enumerate many CI-liaison classes. Let $\mathcal L_{c}(\mathbb P^N)$ denote the set of CI-liaison classes of ACM subschemes of codimension $c$ in $\mathbb P^N$.

\begin{cor} \label{ACM} Let $Y \subseteq{\mathbb P}^{4}$ be a reduced curve  that is not licci locally at the vertex of its affine cone.   Then the ruled join with $Y$ induces a set-theoretic embedding 
$$j_{Y}: \textnormal{ACM}^{\circ}_{c}(\psn)\longhookrightarrow \mathcal L_{c+3}(\mathbb P^{n+5})$$ 
from the set of 
generic complete intersection ACM subschemes of  
  codimension $c$ in $\psn$ to the set of CI-liaison classes of ACM subschemes
of codimension $c+3$ in ${\mathbb P}^{n+5}$.
\end{cor} 
 
\begin{rem}
We can modify the join map to produce non-reduced subschemes 
in smaller dimensions.   Let $J\subseteq k[y_{0},y_{1},y_{2}]$ be an ideal that is not licci locally at the 
irrelevant maximal ideal.  (For example, $J = (y_{0},y_{1},y_{2})^{2}$, cf. \cite[2.1]{HU2}).  Then if $j_{J}$ denotes the map 
taking $X$ to the subscheme defined by the join of the ideal of $X$ and $J$, we have an induced set-theoretic embedding 
$$j_{J}:\textnormal{ACM}^{\circ}_{c}(\psn)\longhookrightarrow \mathcal L_{c+3}(\mathbb P^{n+3}).$$
\end{rem} 

In the special case where we take the join of $J$ (as above) with hypersurfaces of degree $d$ in $\mathbb P^n$, we obtain the following embedding.
\begin{example} For every integer $d\ge 1$, there is an embedding 
$$j_{J}:{\mathbb P}^{\binom{n+d}{d} -1 }\longhookrightarrow \mathcal L_{4}(\mathbb P^{n+3}).$$ 
\end{example}

\section{Liaison and automorphisms} 
In this section we wish to construct linkage classes of ideals $I$ in a power series ring $R$  for which the rings
$R/I$ are all isomorphic, by considering the action via automorphisms.  
 
If $R$ is any ring, we let  Aut$(R)$ denote the group of automorphisms of $R$
and let $G\subseteq $ Aut$(R)$ be a subgroup.  
We denote the action of $g\in G$ on an $R$-ideal $I$ by $gI$.   Since $R/gI\cong R/I$, 
the ideal $gI$ inherits most interesting properties from $I$. 

\begin{lem}
Let $R$ be a Gorenstein local ring and let $I$ and $J$ be unmixed $R$-ideals that are in the same linkage 
class.  Then $gI$ and $gJ$ are in the same linkage class, for any automorphism $g$ of $R$. 
In particular, $$\int gI = g\int I.$$  
\end{lem}
\pf
 By induction, it suffices to show the result when $I$ and $J$ are directly linked.  
In this case, the result follows immediately from the fact that an
automorphism takes complete intersections to complete intersections and preserves ideal quotients.  
\epf

The above lemma shows that any group $G \subseteq $ Aut$(R)$   induces an
action on $\mathcal L_{c}$. 
Now let $I$ be a CM $R$-ideal, and let $[I]$ denote its linkage class.

To have more room to maneuver, as in our earlier study of joins, we again embed the ideal into a flat extension.

\begin{lem}\label{autprop}
Let $(R,m)$ be a local Gorenstein ring with infinite residue field and let 
$I$ be a non-licci $m$-primary
$R$-ideal.  Let $T$ be a flat local Gorenstein extension of $R$ 
with reduced special fiber.   Let $G$ be a group of automorphisms of $T$ such that $mT$ has
trivial stabilizer. 
Then  the orbit map $orb_{I}:G\lra  G\cdot [IT] $  is bijective.
\end{lem} 
\pf  We must show that the stabilizer of $[IT]$ is trivial.   Suppose that there is a sequence
of links $IT =I_{0}\sim \cdots \sim I_{e} = gIT$ joining $IT$ and $gIT$ for some $g\in G$.  
Let $q$ be an associated prime of $mT$.  Then $IT_{q}$ is also not licci \cite[2.12]{AL}, so 
$q$ belongs to the nonlicci locus of $IT$.  Hence
by Proposition~\ref{Nliccidefideal}, $g(IT)\subseteq q$.  Since this holds for every associated prime 
of the special fiber, which is reduced, it follows that $g(IT)\subseteq mT$.  Hence
$I\subseteq IT\subseteq g^{-1}(mT)$, so $I\subseteq g^{-1}(mT)\cap R$.  Since the latter ideal
is reduced, and $I$ is $m$-primary, it follows that $m =  g^{-1}(mT)\cap R$, hence  
$mT\subseteq g^{-1}(mT)$.  Therefore $mT\subseteq g^{-1}(mT) 
\subseteq g^{-2}(mT)\subseteq \cdots$ and hence $g^{-n}(mT) = g^{-n-1}(mT)$ holds for some
$n\ge 0$.  Thus $g(mT) = mT$ and by hypothesis we  conclude that $g = 1$, as required. 
\epf

In order to apply Lemma~\ref{autprop},  
we restrict our attention to the power series ring over a field.  
 
Let $T = k[[X,Y_{1},\ldots,Y_{n}]]$ be a power series ring in $n+1$ variables over a field $k$. 
By the formal inverse function theorem, the group of all $(n+1)\times (n+1)$ 
(lower) unitriangular matrices over $T$ has a representation as a group of automorphisms of $T$,
acting by matrix multiplication on the vector $(X,Y_{1},\ldots,Y_{n})^{t}.$

\begin{prop}\label{aut} Let $R =k[[Y_{1},\ldots,Y_{n}]]$ be a formal power series 
over a field $k$ and   
 let $I$ be a non-licci $m$-primary ideal and let
$T = k[[X,Y_{1},\ldots,Y_{n}]]$.    Then 
there is a natural set-theoretic embedding 
 $$k[[X]]^{n}\longhookrightarrow \mathcal L_{n}$$  
into the set of linkage classes of 1-dimensional $T$-ideals. 
\end{prop} 
\pf  The additive group $G$ of $k[[X]]^{n}$ is represented by the subgroup of  
unitriangular matrices over $(k[[X]])$ 
with nonzero (nondiagonal) elements on the first column.  Clearly  $G$ has a faithful representation as a group of automorphisms of $T$.  
By Lemma~\ref{autprop} it suffices to verify that $G$ acts with trivial stabilizer on $mT$. 
If we denote the action by $$g\cdot Y_{i} = Y_{i} + \xi_{i},$$ then we must show that 
$$(Y_{1} +\xi_{1},\ldots,Y_{n}+\xi_{n}) = (Y_{1},\ldots,Y_{n})$$
only if all $\xi_{i} = 0.$     Since $\xi_{i}\in k[[X]]$, this is clear. 
\epf  
  
Since there are non-licci ideals ($m$-primary) ideals in any codimension $\ge 3$,
the above such embeddings holds for all $n\ge 3$.  
  
In the special case that we consider linear automorphisms, taking the non-licci ideal $m^{2}$ 
(for $n\ge 3$, e.g. \cite{HU2} or Theorem~\ref{PU}), we obtain the following embedding.

\begin{example} For every $n\ge 3$, there is an embedding 
$${\mathbb A}^{n}\longhookrightarrow \mathcal L_{n}(\psn).$$  
\end{example}

\end{document}